\documentclass[12 pt]{amsart}
\usepackage{amsmath,times,epsfig,amssymb,amsbsy,amscd,amsfonts,amstext,graphicx,color,bm,amsthm}
\usepackage[all]{xy}
\usepackage{graphicx}

\usepackage[urlcolor=blue]{hyperref}
\usepackage{tikz,subfigure}

\newcommand{\C}{\mathbb{C}}

\newcommand{\NN}{\mathbb{N}}
\newcommand{\QQ}{\mathbb{Q}}

\newcommand{\Q}{\mathbb{Q}}
\newcommand{\PPP}{\mathbb{P}}

\def\GL{\mathop{\rm  GL}}

\renewcommand{\to}{\longrightarrow}

\newtheorem{Theorem}{Theorem}[section]
\newtheorem{Definition}[Theorem]{Definition}

\newtheorem{Lemma}[Theorem]{Lemma}
\newtheorem{Proposition}[Theorem]{Proposition}
\newtheorem{Corollary}[Theorem]{Corollary}
\newtheorem{Remark}[Theorem]{Remark}
\newtheorem{Example}[Theorem]{Example}
\newtheorem{Conjecture}[Theorem]{Conjecture}

\DeclareMathOperator{\Der}{Der}

\DeclareMathOperator{\pdeg}{pdeg}

\DeclareMathOperator{\codim}{codim}

\def \X(#1){\{x_1,\dots, x_{#1}\}}

\def \GL    {{\rm GL}}

\def \gin   {{\rm gin}}

\newcommand \rgin  {{\rm rgin}}
\newcommand \reg  {{\rm reg}}

\newcommand \M{{\mathcal M}}
\def \LT{{\rm LT}}
\DeclareMathOperator{\rk}{rk}

\newcommand{\A}{\mathcal{A}}

\newcommand \ideal[1] {\langle #1 \rangle}

\begin{document}

\title{New characterizations of freeness for hyperplane arrangements}

\begin{abstract}
In this article we describe two new characterizations of freeness for hyperplane arrangements 
via the study of  the generic initial ideal and of the sectional matrix of the Jacobian ideal of arrangements.
\end{abstract}

\author{Anna Maria Bigatti}
\address{Anna Maria Bigatti, Dipartimento di Matematica, Universit\`a
  degli Studi di Genoa, Via Dodecaneso 35, 16146 Genova , Italy.}
\email{bigatti@dima.unige.it}
\author{Elisa Palezzato}
\address{Elisa Palezzato, Department of Mathematics, Hokkaido University, Kita 10, Nishi 8, Kita-Ku, Sapporo 060-0810, Japan.}
\email{palezzato@math.sci.hokudai.ac.jp}
\author{Michele Torielli}
\address{Michele Torielli, Department of Mathematics, Hokkaido University, Kita 10, Nishi 8, Kita-Ku, Sapporo 060-0810, Japan.}
\email{torielli@math.sci.hokudai.ac.jp}


\date{\today}
\maketitle


\section{Introduction}

Let $V$ be a vector space of dimension $l$ over a field $K$. Fix a system of coordinate $(x_1,\dots, x_l)$ of $V^\ast$. 
We denote by $S = S(V^\ast) = K[x_1,\dots, x_l]$ the symmetric algebra. 
A hyperplane arrangement $\A = \{H_1, \dots, H_n\}$ is a finite collection of hyperplanes in $V$.

Freeness of an arrangement is a key notion which connects arrangement theory with algebraic geometry and combinatorics. 
There are several ways to prove freeness, e.g. using Saito's criterion \cite{saito}, addition-deletion theorem \cite{terao1980arrangementsI}, etc. 
However, is not always easy to characterize freeness. 
In \cite{ziegler1986multiarrangements}, Ziegler proved that the multirestriction $(\A^{H_0}, m^{H_0})$ of a free arrangement $\A$ is also free. 
The converse is not true in general. However in \cite{yoshinaga2004characterization}, Yoshinaga gave a partial converse of Ziegler's 
work and characterized freeness for arrangements by looking at properties around a fixed hyperplane.
In \cite{yoshinaga2005freeness}, Yoshinaga studied arrangements in three-dimensional space and described 
a new characterization of freeness for such arrangements given in terms of the characteristic polynomial and a restricted multiarrangement.
Moreover, with the idea of unifying \cite{yoshinaga2004characterization} and \cite{yoshinaga2005freeness}, 
Schulze in \cite{schulze2012freeness} proved that if the dimension is $l\le4$ (or $l\ge5$ under tameness assumption), 
the freeness of $\A$ is characterized in terms of multirestriction and characteristic polynomials. 
With similar goals, Abe and Yoshinaga in \cite{abe2013free} characterized freeness in terms of the multirestriction and 
the second coefficient of characteristic polynomials (without posing any conditions on dimension or tameness).

The purpose of this paper is to give new characterizations of freeness for any dimension. 
Namely, we characterize freeness in terms of the generic initial ideal and of the sectional matrix of the Jacobian ideal $J(\A)$ of the arrangement $\A$, 
making use of the characterization of Terao \cite{orlterao} of freeness in term of Cohen-Macaulayness
of the Jacobian ideal of the arrangement. 

This paper is organized as follows. 
In \S2, we recall the basic facts about hyperplane arrangements and their freeness. 
In \S3, we describe the connection between the study of free hyperplane arrangements and commutative algebra.
In \S4, we recall the notion of generic initial ideal of a given homogeneous ideal. In \S5, we describe our first characterization via generic initial ideal.  
In \S6, we recall the notion of sectional matrix of a given homogeneous ideal and we prove some new results regarding the sectional matrix. 
In \S7, we describe our second characterization via sectional matrices. In \S8, we describe some additional properties of the generic initial ideal
of the Jacobian ideal of a free hyperplane arrangement.
In \S9,  we reverse our point of view and we describe which strongly stable ideals are rgin of arrangement's Jacobian ideals. 

\section{Preliminares on hyperplane arrangements}\label{sec:arr}

In this section, we recall the terminology and basic notation of hyperplane arrangements and some fundamental results.

Let $K$ be a field of characteristic zero. A finite set of affine hyperplanes $\A =\{H_1, \dots, H_n\}$ in $K^l$ is called a \textbf{hyperplane arrangement}. 
For each hyperplane $H_i$ we fix a defining equation $\alpha_i\in S= K[x_1,\dots, x_l]$ such that $H_i = \alpha_i^{-1}(0)$, 
and let $Q(\A)=\prod_{i=1}^n\alpha_i$. An arrangement $\A$ is called \textbf{central} if each $H_i$ contains the origin of $K^l$. 
In this case, the defining equation $\alpha_i\in S$ is linear homogeneous, and hence $Q(\A)$ is homogeneous of degree $n$. 

Let $L(\A)=\{\bigcap_{H\in\mathcal{B}}H \mid \mathcal{B}\subseteq\A\}$ be the \textbf{lattice of intersection} of $\A$. 
Define a partial order on $L(\A)$ by $X\le Y$ if and only if $Y\subseteq X$, for all $X,Y\in L(\A)$. 
Note that this is the reverse inclusion. Define a rank function on $L(\A)$ by $\rk(X)=\codim(X)$. 
$L(\A)$ plays a fundamental role in the study of hyperplane arrangements, in fact it determines the combinatorics of the arrangement.
Let $L^p(\A)=\{X\in L(\A)~|~\rk(X)=p\}$. We call $\A$ \textbf{essential} if $L^l(\A)\ne\emptyset$.
We denote by $\Der_{K^l} =\{\sum_{i=1}^l f_i\partial_{x_i}~|~f_i\in S\}$ the $S$-module of \textbf{polynomial vector fields} on $K^l$ (or $S$-derivations). 
Let $\delta =  \sum_{i=1}^l f_i\partial_{x_i}\in \Der_{K^l}$. Then $\delta$ is  said to be \textbf{homogeneous of polynomial degree} $d$ if $f_1, \dots, f_l$ are homogeneous polynomials of degree~$d$ in $S$. 
In this case, we write $\pdeg(\delta) = d$.

\begin{Definition} 
Let $\A$ be a central arrangement. Define the \textbf{module of vector fields logarithmic tangent} to $\A$ (logarithmic vector fields) by
$$D(\A) = \{\delta\in \Der_{K^l}~|~ \delta(\alpha_i) \in \ideal{\alpha_i} S, \forall i\}.$$
\end{Definition}

The module $D(\A)$ is obviously a graded $S$-module and we have that $D(\A)= \{\delta\in \Der_{K^l}~|~ \delta(Q(\A)) \in \ideal{ Q(\A)} S\}$. 
In particular, since the arrangement $\A$ is central, then the Euler vector field $\delta_E=\sum_{i=1}^lx_i\partial_{x_i}$ belongs to $D(\A)$. 
In this case, we can write $D(\A)\cong S{\cdot}\delta_E\oplus D_0(\A)$, where $D_0(\A)=\{\delta\in \Der_{K^l}~|~ \delta(Q(\A))=0\}$.

The following is the more used definition of a free hyperplane arrangement. However in the rest of the paper we will use as a definition
the equivalence described in Theorem~\ref{theo:freCMcod2}.
\begin{Definition} 
A central arrangement $\A$ is said to be \textbf{free with exponents $(e_1,\dots,e_l)$} 
if and only if $D(\A)$ is a free $S$-module and there exists a basis $\delta_1,\dots,\delta_l \in D(\A)$ 
such that $\pdeg(\delta_i) = e_i$, or equivalently $D(\A)\cong\bigoplus_{i=1}^lS(-e_i)$.
\end{Definition}


\begin{Remark}\label{rem:assumptexpfree}
Let $\A$ be free with exponents $(e_1,\dots,e_l)$. We can suppose that $e_1\le e_2\le\cdots\le e_l$. 
Moreover, if $\A$ is essential then $e_1=1$. 
\end{Remark}

%

\section{Hyperplane arrangements and commutative algebra}

The purpose of this paper is to study free hyperplane arrangements in the language of commutative algebra.
For this reason, we start our investigation from the characterization of freeness described by Terao that connects exactly 
the theory of hyperplane arrangements with commutative algebra, see \cite{orlterao}.

\begin{Definition} Given an arrangement $\A=\{H_1,\dots, H_n\}$ in $K^l$, the \textbf{Jacobian ideal} of $\A$
is the ideal of $S$ generated by $Q(\A)$ and all its partial derivatives, and it is denoted by $J(\A)$.
\end{Definition}
Notice that, since $J(\A)$ is the ideal describing the singular locus of $\A$, we have that $S/J(\A)$ is $0$ or $(l-2)$-dimensional. 
\begin{Remark} Let $\A$ be a central arrangement. Then $Q(\A)$ is homogenous and hence we can write $$nQ(\A)=\sum_{i=1}^lx_i\frac{\partial Q(\A)}{\partial x_i}$$
This implies that if $\A$ is central, then $J(\A)$ is a homogeneous ideal generated by at most $l$ polynomials all of degree $n{-}1$.
\end{Remark}


\begin{Theorem}[Terao's criterion]\label{theo:freCMcod2} 
A central arrangement $\A$ is free if and only if $S/J(\A)$ is $0$ or Cohen-Macaulay.
\end{Theorem}

If an arrangement $\A$ is free, then we can easily compute the minimal resolution of the Jacobian ideal. See \cite{orlterao} for more details.
\begin{Remark}\label{remark:resjaccmfree} Let $\A$ be a central, essential and free hyperplane arrangement 
with exponents $(e_1,\dots, e_l)$. Then, by Hilbert-Burch Theorem, $J(\A)$ has a minimal free resolution of the type
\begin{equation*}\label{eq:resfreejac}
0\to\bigoplus_{i=2}^lS(-n-e_i+1)\cong D_0(\A)\to S(-n+1)^l\to J(\A)\to0.
\end{equation*}
\end{Remark}

%
%
%

\section{Generic initial ideal}\label{sec:GIN}

In this section we recall the definition and some known properties of the generic initial ideal.
We also present a new result which is the starting point of
our first characterization in Section~\ref{sec:freearrgin}.




\begin{Definition} 
A monomial ideal $B$ in $K[x_1,\dots, x_l]$ is said to be  \textbf{strongly stable} if for every power-product $t \in B$ and
every $i,j$ such that $i<j$ and $x_j|t$, the power-product $x_i\cdot t/x_j$ is in $B$.
\end{Definition}

Directly from the definition of strongly stable ideal, 
we have the following lemma.

\begin{Lemma}\label{lem:gingensequiv} 
Let $B$ be a strongly stable ideal in $K[x_1, \ldots , x_l]$ and $k\in\{1,\dots, l\}$. Then 
$B$ has no minimal generators divisible by $x_k$ if and only if $B$ has no minimal generators divisible by $x_k,\dots, x_l$.
\end{Lemma}

\begin{Definition} 
Let $\sigma$ be a term ordering on $S=K[x_1,\dots,x_l]$ and $f$ a non-zero polynomial in $S$. 
Then $\LT_\sigma(f)= \max_\sigma\{{\rm Supp}(f)\}$.  If
$I$ is an ideal in $S$, then the 
\textbf{leading term ideal} (or \textit{initial ideal})
of $I$ is the ideal $\LT_\sigma(I)$ generated by
 $\{\LT_\sigma(f) \mid f \in I{\setminus}\{0\}\;\}$.
\end{Definition}

The following theorem is due to Galligo \cite{galligo1974propos}.

\begin{Theorem}[Galligo]\label{thm:Galligo}
Let $I$ be a homogeneous ideal in $K[x_1,\dots,x_l]$, with $K$ a
field of characteristic $0$ and $\sigma$ a term ordering such that
$x_1>_\sigma x_2 >_\sigma \dots >_\sigma x_l$. Then there exists a Zariski
open set $U\subseteq\GL(l)$ and a strongly stable ideal $B$ such that for each  $g\in U$, $\LT_\sigma(g(I)) = B$.
\end{Theorem}

\begin{Definition} 
The strongly stable ideal $B$ given in Theorem~\ref{thm:Galligo} is called the \textbf{\boldmath generic initial ideal with respect to $\sigma$} 
of $I$ and it is denoted by {\boldmath$\gin_\sigma(I)$}.  
In particular, when $\sigma=$DegRevLex, $\gin_\sigma(I)$ is simply denoted with {\boldmath$\rgin(I)$}.
\end{Definition}

Since we are interested in studying free hyperplane arrangements, we need the following result on Cohen-Macaulay ideals by Bayer and Stillman \cite{bayer1987criterion}.
\begin{Theorem}\label{theo:ginCM} 
Let $I$ be a homogeneous ideal in $S=K[x_1, \dots, x_l]$. 
Then $I$ is Cohen-Macaulay if and only if $\rgin(I)$ is Cohen-Macaulay.
Moreover, a regular sequence
  for $S/\rgin(I)$ is  $x_l,x_{l-1}, \dots, x_{l-\dim(S/I)+1}$.
\end{Theorem}

We now mention some results about the degree of the generators in
$\rgin(I)$, concluding with a new corollary.
 In particular, our goal is to characterize the $\rgin$ associated to
 a free hyperplane arrangement in terms of its generators
(Theorem~\ref{theo:firstequivfregin}).

\begin{Remark}\label{rem:genrgin}
Let $I$ be a homogeneous ideal in $S$.
If $I$ has a minimal generator of degree~$d$, then also $g(I)$ does and then
$\rgin(I)$ has a minimal generator of degree~$d$. 
The converse is not true in general: consider for example the ideal  $I=\ideal{z^5,xyz^3}$ in $\Q[x,y,z]$, whose $\rgin(I)$ is $\ideal{x^5,x^4y,x^3y^3}$.
\end{Remark}

\begin{Definition} 
Let $I$ be a homogeneous ideal in the ring $K[x_1,\dots,x_l]$. 
The \textbf{Castelnuovo-Mumford regularity} of $I$, denoted~{\boldmath$\reg(I)$},
 is the maximum of the numbers $\beta_{i,j}(I)-i$, 
where $\beta_{i,j}(I)$ are the graded Betti numbers of $I$.
\end{Definition}


\begin{Theorem}[\cite{bayer1987criterion}]\label{theo:regginequalideal} 
Let $I$ be a homogeneous ideal in $K[x_1,\dots,x_l]$. Then $\reg(I)=\reg(\rgin(I))$. 
Moreover, if $B$ is a strongly stable ideal, then $\reg(B)$ is the highest degree of a minimal generator of $B$.
\end{Theorem}

\begin{Lemma}[\cite{BPT2016} Lemma 4.4]\label{lemma:no_gens_i}
Let $I$ be a homogeneous ideal in the ring $S=K[x_1,\dots,x_l]$ 
generated  in
 degree $\le D$. 
If there exists $i\le l$  such that
$\rgin(I)$ has no minimal generators of degree $D$ in
$S_{(i)}=K[x_1,\dots,x_i]$,
then $\rgin(I)$ has no minimal generators of any degree $\ge D$ in $S_{(i)}$.
\end{Lemma}

\begin{Corollary}\label{cor:genofalldegmax} 
Let $I$ be a homogeneous ideal in $S=K[x_1,\dots,x_l]$
and let $D$ be the highest degree of a minimal generator of $I$.
 Then $\rgin(I)$ has at least one minimal generator of degree~$d$, for all $d\in\{D,\dots,\reg(I)\}$.
\end{Corollary}

\begin{proof} 
If $I=S$, then the Corollary is trivially true.

Suppose now $I\subsetneq S$.
By Remark~\ref{rem:genrgin}, $\rgin(I)$ has a minimal generator
of degree~$D$, and, by Theorem~\ref{theo:regginequalideal}
$\reg(I)$ is the highest degree of a minimal generator of $\rgin(I)$.

Now, by Lemma~\ref{lemma:no_gens_i} for $i=l$ we know that
if $\rgin(I)$ has no minimal generators of degree $d>D$
then $\rgin(I)$ has no minimal generators of any degree $\ge d$.
Thus we conclude that $d> \reg(I)$.
\end{proof}

\section{Hyperplane arrangements and generic initial ideals}\label{sec:freearrgin}

In this section we present our first characterization of freeness for a central hyperplane arrangement $\A$ in $K^l$, 
where $K$ is a field of characteristic zero. 
We characterize freeness by looking at the generic initial ideal of the Jacobian ideal $J(\A)$ of $\A$.

\medskip

Before presenting our first characterization, we describe some of the properties of $\rgin(J(\A))$ without assuming $\A$ to be free.

Since $J(\A)$ is a homogeneous ideal generated by $l$ polynomials of degree $n{-}1$,
the following lemma is just a rewriting of Corollary
\ref{cor:genofalldegmax} for the case $I=J(\A)$.

\begin{Lemma}\label{lemm:genofalldeg} 
Let $\A=\{H_1,\dots, H_n\}$ be a central arrangement in $K^l$. Then
$\rgin(J(\A))$ has at least one minimal generator of degree~$d$, for
all $d\in\{n{-}1,\dots,\reg(J(\A))\}$, 
and no minimal generator outside that range.
\end{Lemma}
\begin{Remark} 
Notice that in general 
$\rgin$ of a homogeneous ideal may be generated in non-consecutive degrees.
For example $B=(x^2,xy,y^5)\subset K[x,y]$ is strongly stable,
thus $\rgin(B) = B$ has minimal generators only in degree $2$ and $5$.
\end{Remark}


\begin{Lemma}\label{lemm:yalwaysgen} 
Let $\A$ be a central arrangement in $K^l$. Then there exist $\alpha\ge 1$ such that $x_2^\alpha\in\rgin(J(\A))$. 
In other words, using the language of Section~\ref{sec:SectionalMatrix}, the $(l-2)$-reduction number $r_{l-2}(S/J(\A))$ is finite.
\end{Lemma}


\begin{proof} 
By construction,
 $S/J(\A)$ is $0$ or $(l-2)$-dimensional,
thus $\rgin(J(\A))$ is $J(\A) = S$ or it must contain some powers of $x_1$
  and $x_2$ because it is strongly
  stable. Hence, in either case, there exists a positive power of $x_2$
in $\rgin(J(\A))$.
The second part of the statement then follows from 
the Definition \ref{def:reductnumb} of the $2$-reduction number in terms of the
sectional matrix.
\end{proof}


We are now ready to present our first characterization.
\begin{Theorem}\label{theo:firstequivfregin} 
Let $\A =\{H_1, \dots, H_n\}$ be a central arrangement in $K^l$. 
Then $\A$ is free if and only if $\rgin(J(\A))$ is $S$ or 
its minimal generators include $x_1^{n-1}$, some positive power of
$x_2$, and no monomials in $x_3,\dots, x_l$. 


More precisely, if $\A$ is free, then 
$\rgin(J(\A))$ is $S$ or it is minimally generated by
$$x_1^{n-1},\; x_1^{n-2}x_2^{\lambda_1},\; \dots,\; x_2^{\lambda_{n-1}}$$ 
with $1\le\lambda_1<\lambda_2<\cdots<\lambda_{n-1}$ and $\lambda_{i{+}1}-\lambda_i= 1$ or $2$.
\end{Theorem}

\begin{proof} 
By Theorem~\ref{theo:freCMcod2}, $\A$ is free if and only if $S/J(\A)$ is $0$ or $(l{-}2)$-dimensional Cohen-Macaulay. 
Clearly, the ring $S/J(\A)$ is $0$ 
if and only if $\rgin(J(\A))=S$.
Suppose now that $J(\A)\subsetneq S$. 
Since $J(\A)$ is an ideal generated by $l$ homogenous polynomials of degree $n{-}1$, then $x_1^{n-1}$ is a minimal generator of
$\rgin(J(\A))$. By Theorem~\ref{theo:ginCM}, $J(\A)$ is Cohen-Macaulay of codimension 2 if and only if $\rgin(J(\A))$ is Cohen-Macaulay of codimension 2, 
and this is equivalent to $\rgin(J(\A))$ having a power of $x_2$ as
minimal generator and no minimal generators in $x_3,\dots, x_l$. 

Under these constrains, 
the only possible strongly stable ideals are the lex-segment ideals, minimally
generated by
 $x_1^{n-1},x_1^{n-2}x_2^{\lambda_1},\dots,x_2^{\lambda_{n-1}}$, 
with $1\le\lambda_1<\lambda_2<\cdots<\lambda_{n-1}$.
Notice that there must be exactly one generator for each power of $x_1$ from $n-1$ to $0$, 
so there are exactly $n = \#\A$ generators.
Finally, if $\A$ is free, from Lemma~\ref{lemm:genofalldeg} we know that there are no ``holes'' in the
sequence of the degrees of the minimal generators, and this translates into
$\lambda_{i{+}1}-\lambda_i= 1$ (same degree) or $2$ (consecutive degrees).
\end{proof}
\begin{Example}\label{ex:recurring}
Consider the arrangement $\A$ in $\C^3$ defined by the equation $Q(\A)=xyz(x+y)(x-y)$.
Then the generic initial ideal of its Jacobian ideal is $\ideal{x^4,x^3y,x^2y^2,xy^4,y^6}$ and hence $\A$ is free.

Similarly consider the arrangement $\A$ in $\C^3$ defined by the equation $Q(\A)=x(x+y-z)(x+z)(x+2z)(x+y+z)$.
Then the generic initial ideal of its Jacobian ideal is $\ideal{x^4, x^3y, x^2y^2, xy^4, y^5, xy^3z^2}$. Since $z$ divides 
 a minimal generator of $\rgin(J(\A))$, then $\A$ is not free.
\end{Example}

\begin{Remark} By Lemma~\ref{lemm:yalwaysgen}, the previous theorem is a new proof of the known fact that any central line arrangement in the plane is free.
\end{Remark}

%
%

We conclude the section with a conjecture about the generic initial ideal of a central arrangement not necessarily free.

\begin{Conjecture}\label{conj:generatZ} Let $\A =\{H_1, \dots, H_n\}$
  be a central arrangement in $K^l$, and $d_0=\min\{d~|~x_2^{d+1}\in\rgin(J(\A))\}$. 
  If $\rgin(J(\A))$ has a minimal generator $T$ that involves the third variable of $S$, then $\deg(T)\ge d_0+1$.
\end{Conjecture}

\begin{Example}
In Example~\ref{ex:recurring} we had a non free arrangement whose
$\gin$ is $\ideal{x^4, x^3y, x^2y^2, xy^4, y^5, xy^3z^2}$,
and we observe that $\deg(xy^3z^2) = 6>5 = \deg(y^5)$.
The previous statement is false in general 
as shown by the strongly stable ideal
$B = \rgin(B) = \ideal{x^2, xy, xz, y^3}$.
\end{Example}

\section{Sectional matrix}\label{sec:SectionalMatrix}

The definition of the Hilbert function of a homogenous ideal in $S$ was extended in \cite{bigatti1997borel} to
the definition of \textbf{the sectional matrix}: the bivariate function encoding the Hilbert functions of the generic hyperplane sections. 
In this section, we recall the definition and basic properties of the sectional matrix for the quotient algebra $S/I$, as described in \cite{BPT2016}.
Then we present some new results that will play an important role in the characterization of Section \ref{sec:freearrsectmat}.

\begin{Definition}\label{def:SM} 
Given a homogeneous ideal $I$ in $S=K[x_1,\dots,x_l]$, the
\textbf{sectional matrix} 
of $S/I$ is the function
$\{1,\dots,l\}\times \NN \longrightarrow \NN$
\begin{eqnarray*}
\M_{S/I}(i,d) &=& \dim_K(S_d/(I+(L_1,\dots,L_{l-i}))_d),
\end{eqnarray*}
where $L_1,\dots,L_{l-i}$ are generic linear forms.
\end{Definition}

The following result reduces the study of the sectional matrix of
a homogeneous ideal to the combinatorial behaviour of a monomial ideal.

\begin{Theorem}[Lemma~5.5, \cite{bigatti1997borel}]\label{theo:rgin}
Let $I$ be a homogeneous ideal in $S = K[x_1,\dots, x_l]$. 
Then 
$$\M_{S/I} (i,d) = \M_{S/\rgin(I)}(i,d) =
\dim_K(S_d/(\rgin(I)+(x_{i+1},\dots,x_l))_d).$$
\end{Theorem}

\begin{Remark}\label{rem:stronglystable}
Theorem~\ref{theo:rgin} shows that when we have a strongly stable
ideal $B\in S$ (and in particular $\rgin(I)$ is strongly stable) the
sectional matrix of ${S/B}$ is particularly easy to compute because
sectioning $B$ by $l{-}i$ generic linear forms is the same as
sectioning $B$ by the smallest $l{-}i$ indeterminates, $x_{i{+}1},\dots, x_l$.
\end{Remark}

The following results show, for a strongly stable ideal $B$, the link between having no generators and a recurrence in the sectional matrix.

\begin{Proposition}[\cite{BPT2016}]\label{prop: BReq_rgingen} 
Let $B$ be a strongly stable ideal in the polynomial ring $S=K[x_1, \ldots, x_l]$. 
Then $\M_{S/B}(i,d{+}1) =\sum\limits_{j=1}^i
\M_{S/B}(j,d)$ if and only if 
$B$ has no minimal generators in degree $d{+}1$ in $x_1,\dots, x_i$.
\end{Proposition}

\begin{Theorem}\emph{(Theorem 4.5, \cite{bigatti1997borel})}
\label{thm:BR_Persistence_Bor}
Let $B$ be a strongly stable ideal in $S=K[x_1,\dots, x_l]$ with generators of degree $\le D$. Then
\begin{enumerate}
\item\label{BR_PersistenceSum_Bor} $\M_{S/B}(i, d{+}1) = \sum_{j=1}^i \M_{S/B}(j,d)$ 
for all $d\ge D$ and $i=1,\dots,l$.
\item\label{BR_PersistanceSum} $\M_{S/B}(i,d{+}1)=\M_{S/B}(i{-}1,d{+}1)+\M_{S/B}(i,d)$,
for all $d\ge D$ and $i=1,\dots, l$.
\end{enumerate}
\end{Theorem}



The equality in Theorem~\ref{thm:BR_Persistence_Bor}.(1)
  was then developed into an inequality for homogeneous ideals and
  investigated in \cite{bigatti1997borel} and \cite{BPT2016}.  In this
  paper we develop and exploit the equality in \ref{thm:BR_Persistence_Bor}.(2) (see
  Theorem~\ref{theo:sectmatrixgenxk} below).  

%
%
%



\bigskip

The remaining of this section is devoted to introduceing some new results
on sectional matrices and generic initial ideals. These results are the keys for our second characterization of
freeness for hyperplane arrangement, see
Theorem~\ref{theo:secondcharacterizfree}. In particular, our goal is
to identify the \textit{minimal} number of entries we need to check in
the sectional matrix to ensure that the given ideal is Cohen-Macaulay.



\begin{Theorem}\label{theo:sectmatrixgenxk}
Let $I$ be a non-zero homogeneous ideal in the polynomial ring $S=K[x_1, \ldots , x_l]$, $i\in\{1,\dots, l\}$ and $d\ge1$.
Then $$\M_{S/I}(i,d)\le \M_{S/I}(i{-}1,d)+\M_{S/I}(i,d{-}1).$$
Moreover, the equality holds if and only if $\rgin(I)$ has no minimal generator of degree~$d$ divisible by $x_i$.
\end{Theorem}
\begin{proof} 
Without loss of generality, we may assume $I=B$ strongly stable, because
$\M_{S/I}=\M_{S/\rgin(I)}$, by Theorem~\ref{theo:rgin},
and also because
any strongly stable ideal $B$ coincides with its
$\rgin$.

For the first part of the statement, we start observing that for any ideal $I'$ we have that
 $I'_d\cap K[x_1,\dots, x_i]$ must contain all the elements of $I'_d\cap
K[x_1,\dots, x_{i-1}]$ and all the elements of $I'_{d-1}\cap
K[x_1,\dots, x_i]$ multiplied by $x_i$, notice that the last two sets
are disjoint.
So it follows that
$$\dim_K(I'_d\cap K[x_1,\dots, x_i])$$
$$\ge\dim_K(I'_d\cap K[x_1,\dots, x_{i-1}])+\dim_K(I'_{d-1}\cap K[x_1,\dots, x_i]).$$

Then the desired inequalities follow from 
Theorem~\ref{theo:rgin} and
$$\M_{S/B}(i,d) = 
\dim_K(K[x_1,\dots, x_i])-\dim_K(B_d\cap K[x_1,\dots,x_i])$$
$$\le\M_{S/B}(i-1,d)+\M_{S/B}(i, d{-}1).$$

For the second part of the statement, suppose the equality holds:
then $\dim_K(B_d\cap K[x_1,\dots, x_i])
=\dim_K(B_d\cap K[x_1,\dots, x_{i-1}])
+\dim_K(B_{d-1}\cap K[x_1,\dots, x_i])$ 
and this implies that $B$
has no minimal generator of degree~$d$ divisible by~$x_i$.

On the other hand, suppose that $B$ has no minimal generator of
degree~$d$ divisible by $x_i$ and let $t$ be a power-product in 
$B_d\cap K[x_1,\dots,x_i]$.
If $x_i$ does not divide~$t$, then $t\in B_d\cap K[x_1,\dots,x_{i-1}]$.
Otherwhise
$t= x_i\cdot t'$.
We claim $t'\in B_{d-1}\cap K[x_1,\dots, x_i]$.
By hypothesis $t$ cannot be a minimal generator and so 
$t=x_j\cdot t''$ for some $j\in\{1,\dots, i\}$ 
and $t''\in B_{d-1}\cap K[x_1,\dots, x_i]$.
But $B$ is strongly stable, and so
$t'=x_j\cdot t''/x_i\in B_{d-1}\cap K[x_1,\dots, x_i]$,
as we claimed. 
This implies that $\dim_K(B_d\cap K[x_1,\dots, x_i])=\dim_K(B_d\cap K[x_1,\dots, x_{i-1}])+\dim_K(B_{d-1}\cap K[x_1,\dots, x_i])$ 
and hence $\M_{S/B}(i,d)= \M_{S/B}(i-1,d)+\M_{S/B}(i,d-1)$.
\end{proof}


The equality in Theorem~\ref{thm:BR_Persistence_Bor}.(1),
occurring for a homogeneous ideal,
was called in \cite{BPT2016} \textit{$i$-maximal growth in degree $d$}.
The equality in Theorem~\ref{theo:sectmatrixgenxk}
is weaker (see Example~\ref{ex:first}), 
and is crucial in this paper, so we give it a name.

\begin{Definition} Let $I$ be a non-zero homogeneous ideal in the polynomial ring $S=K[x_1, \ldots , x_l]$, $i\in\{2,\dots, l\}$ and $d\ge1$.
We say that $\M_{S/I}$ 
\textbf{\boldmath has the triangle equality in position $(i,d)$}
if and only if
$$\M_{S/I}(i,d)= \M_{S/I}(i{-}1,d)+\M_{S/I}(i,d{-}1).$$
\end{Definition}


\begin{Example}\label{ex:first}
By the description in \cite{BPT2016}, if
$\M_{S/I}(i, d) = \sum_{j=1}^i \M_{S/I}(j,d{-}1)$\\
then we have
$\M_{S/I}(i,d)= \M_{S/I}(i{-}1,d)+\M_{S/I}(i,d{-}1).$

The opposite implication is false.
Let $S=\QQ[x,y,z]$ 
and $I = \ideal{x^4 {-}y^2z^2,$ $ xy^2 {-}yz^2 {-}z^3}$ an ideal of $S$. Then the sectional matrix of $S/I$ is
$$
\begin{array}{rcccccccccc} 
            _0 & _1 & _2 & _3  & _4 
          & _5 & _6 & _7 & \dots\\
  1 & 1 & 1 & 0 & 0 & 0 & 0 & 0 & \dots\\
  1 & 2 & 3 & 3 & 2 & 1 & 0 & 0 & \dots\\
  1 & 3 & 6 & 9 & \fbox{$11$} & 12 & 12 & 12 & \dots
\end{array} 
$$
If we consider $i{=}3$ and $d{=}4$, 
then $\M_{S/I}(3,4)=\M_{S/I}(2,4)+\M_{S/I}(3,3)$, but
$\M_{S/I}(3,4)<\sum_{s=1}^3\M_{S/I}(s,3)$.
Indeed, $\rgin(I) = \ideal{x^3,  x^2y^2,  xy^4,  y^6}$ has no minimal
generator divisible by $z$,
so the triangle equality holds in the whole $3$rd row.
\end{Example}

In the case of a homogeneous ideal, putting together Theorem
\ref{theo:sectmatrixgenxk} and Lemma~\ref{lem:gingensequiv},
 we have the following corollary showing that a finite number of equalities
in the $k$-th row implies the equalities hold also for 
each and whole $s$-th row, with $s\ge k$.

\begin{Corollary}\label{corol:rowcolumnsum} 
Let $I$ be a non-zero homogenous ideal in the polynomial ring $S=K[x_1, \ldots , x_l]$ and $i\in\{2,\dots, l\}$. Then the following facts are equivalent
\begin{enumerate}
\item\label{singequal3} $\M_{S/I}$ has the triangle equality in position $(i,d)$
for all $d\le \reg(I)$.
\item\label{allequal3} $\M_{S/I}$ has the triangle equality in position $(s,d)$
for all $d\in\NN$ and $s\ge i$.
\end{enumerate}
\end{Corollary}

\begin{proof} 

Clearly (\ref{allequal3}) implies  (\ref{singequal3}). On the other
hand, by Theorem~\ref{theo:regginequalideal}
$\rgin(I)$ has no minimal generator of degree $>\reg(I)$, and
by Theorem~\ref{theo:sectmatrixgenxk}, Claim~(\ref{singequal3})
implies that $\rgin(I)$ has no minimal generator divisible by
$x_i$ for all $d\le\reg(I)$. 
Hence, by Lemma~\ref{lem:gingensequiv}, $\rgin(I)$ has no
minimal generators divisible by $x_i,\dots, x_l$, 
and we conclude by applying again Theorem~\ref{theo:sectmatrixgenxk}.
\end{proof}

The definition of $s$-reduction number has several equivalent
formulations and we recall here the one given in  
\cite{BPT2016}.

\begin{Definition}\label{def:reductnumb}
Given $I$ a homogeneous ideal in $S= K[x_1,\dots,x_l]$, we define
the \textbf{$i$-reduction number} of $S/I$ as 
$$r_i(S/I) = \max\{d \mid \M_{S/I}(l{-}i,d) \ne 0\},$$
or, equivalently, $r_i(S/I) =\min\{d\mid x_{l-i}^{d+1}\in\rgin(I)\}.$
\end{Definition}

Now we apply these results to the Cohen-Macaulay case.

\begin{Theorem}\label{theo:charcmhomideal}  
Let $I$ be a non-zero homogeneous ideal in $S=K[x_1, \ldots , x_l]$. 
Then $S/I$ is Cohen-Macaulay of codimension $i$ if and only if 
the following two conditions hold
\begin{enumerate}
\item $d_0=r_{l-i}(S/I)$ is finite,
\item \label{theo:charcmhomideal.triangle}
$\M_{S/I}$ has the triangle equality in position $(i{+}1,d)$ for all $d \le\reg(I)$.
\end{enumerate}
\end{Theorem}


\begin{proof} 
By Theorem~\ref{theo:ginCM}, $I$ is Cohen-Macaulay of codimension $i$
if and only if $\rgin(I)$ is Cohen-Macaulay of codimension $i$.
Having $\M_{S/I} = \M_{S/\rgin(I)}$ by Theorem~\ref{theo:rgin}, we may
assume $I=B$, a strongly stable ideal.

Being $B$ strongly stable, $x_k^\alpha\in B$ implies
$x_j^\alpha\in B$ for all $j\le k$, so~$B$ is Cohen-Macaulay of codimension $i$
if and only if  there exists $\alpha_i\ge 1$
such that $x_i^{\alpha_i}$ is a minimal generator $B$,
in particular $\reg(B)=\alpha_i$,
and  $x_{i+1},\dots,x_l$ is an $S/B$-regular sequence,
\textit{i.e.}
 no minimal generator of $B$ is divisible by $x_s$ with $s>i$.

 In terms of sectional matrix, such $\alpha_i$ exists if and only if
 $\M_{S/B}(i,d)=0$ for all $d\ge \alpha_i$, in other words, if and
 only if $d_0$ is finite.

Moreover, the equality in (\ref{theo:charcmhomideal.triangle}) for the
$i{+}1$ row, and $d\le\reg(B)$, is
equivalent, by Corollary~\ref{corol:rowcolumnsum},  to the equality
for each $s$ row with $s\ge i+1$, and for all degrees.
And this is equivalent, 
by Theorem~\ref{theo:sectmatrixgenxk}, 
to $B$ having no minimal generators divisible by $x_s$ with $s>i$.
\end{proof}
\begin{Remark}\label{rem:newequalrowssum} By the Definition \ref{def:SM} of the sectional matrix, it follows that $\M_{S/I}$ has the triangle equality in position $(i{+}1,d)$ for all $d \le\reg(I)$ if and only if
\begin{equation*}\label{eq:newequalrowssum}\M_{S/I}(i+1,\reg(I))=\sum_{d=0}^{\reg(I)}\M_{S/I}(i,d).
\end{equation*}
\end{Remark}

The following example shows how easily we can visualize the previous theorem. 

\begin{Example}\label{exCMnonCM}
Consider the ring $S=\QQ[x,y,z,w]$ and the ideal 
$I=\ideal{xz,yw}\cap\ideal{x+z,xy}$ of $S$. 
Clearly $S/I$ is Cohen-Macaulay of codimension~$2$. 
In fact, $\reg(I)=3$, $d_0=2$, and the sectional matrix of $S/I$ is
$$
\begin{array}{rcccccccccc} 
            _0 & _1 & _2 & _3  & _4 & \dots\\
  1 & 1 & 1 & 0 & 0 &  \dots\\
  1 & 2 & 3 & 0 & 0 &  \dots\\
  1 & \fbox{$3$} & \fbox{$6$} &  \fbox{$6$} &  6 &  \dots \\
  1 & 4 & 10 & 16 & 22 & \dots
\end{array} 
$$
with the $0$ in the second row and the triangular equality in the third one.

If we consider the ideal $J_1=\ideal{x}\cap\ideal{xz,yw}$ of $S$, 
then $S/J_1$ has dimension $3$ but it is not Cohen-Macaulay. 
In fact, $\reg(I)=3$, $d_0=1$ and the sectional matrix of $S/J_1$ is 
$$
\begin{array}{rcccccccccc} 
            _0 & _1 & _2 & _3  & _4 & \dots\\
  1 & 1 & 0 & 0 & 0 &  \dots\\
  1 & \fbox{$2$} & \fbox{$2$} & \fbox{\fbox{$1$}} & 1 &  \dots\\
  1 & 3 & 5 &  6 &  7 &  \dots \\
  1 & 4 & 9 & 15 & 22 & \dots
\end{array} 
$$
If we consider the ideal $J_2=\ideal{x^2,xy^2,xyz,y^4}$ of $S$, then $S/J_2$ has dimension $2$ but it is not Cohen-Macaulay. 
In fact, $\reg(I)=4$, $d_0=3$ and the sectional matrix of $S/J_1$ is
$$
\begin{array}{rcccccccccc} 
            _0 & _1 & _2 & _3  & _4 & _5 & \dots\\
  1 & 1 & 0 & 0 & 0 & 0 &  \dots\\
  1 & 2 & 2 & 1 & 0 & 0 & \dots\\
  1 & \fbox{$3$} & \fbox{$5$} & \fbox{\fbox{$5$}} & \fbox{$5$} & 5 & \dots \\
  1 & 4 & 9 & 14 & 19 & 24 & \dots
\end{array} 
$$
and we can see that $5=\M_{S/J_2}(3,3)\ne\M_{S/J_2}(3,2)+\M_{S/I}(2,3)=5+1$.
\end{Example}

\section{Hyperplane arrangements and sectional matrices}
\label{sec:freearrsectmat}

In this section we present our second characterization of freeness for central hyperplane arrangements in $K^l$, 
where $K$ is a field of characteristic zero. 
We characterize freeness by looking at the sectional matrix of $S/J(\A)$.

\begin{Theorem}\label{theo:secondcharacterizfree}
Let $\A$ be a central arrangement and $d_0=r_{l-2}(S/J(\A))$. 
Then $\A$ is free if and only if $\M_{S/J(\A)}$ is the zero function or 
the following two conditions hold
\begin{enumerate}
\item $\M_{S/J(\A)}(3,d_0)=\M_{S/J(\A)}(3,d_0{+}1)=\M_{S/J(\A)}(3,d_0+2)$,
\item $\M_{S/J(\A)}(3,d_0)=\sum_{d=0}^{d_0}\M_{S/J(\A)}(2,d)$,
or, equivalently,
$\M_{S/J(\A)}$ has the triangle equality in position $(3,d)$, 
for all $2\le d\le d_0$.
\end{enumerate}

\end{Theorem}

%
\begin{proof} 
By Theorem~\ref{theo:freCMcod2}, $\A$ is free if and only if $S/J(\A)$ is $0$ or $(l{-}2)$-dimensional Cohen-Macaulay. 
Clearly, $S/J(\A)$ is zero if and only if $\M_{S/J(\A)}$ is the zero function.

Suppose now that $S/J(\A)$ is non-zero.
Let $B = \rgin(J(\A))$ and recall that $\M_{S/B} = \M_{S/J(\A)}$,
and $\reg(B) = \reg(S/J(\A))$.
 From Lemma~\ref{lemm:yalwaysgen} we have that,
being $\A$ a central arrangement, $d_0$ is
finite and $x_2^{d_0{+}1}$ is a minimal generator of $B$.
Let $\A$ be free, then
by Theorem~\ref{theo:charcmhomideal}, 
$\M_{S/B}$ has the triangle equality in position $(3,d)$ for all
$d \le\reg(B)$, and $\reg(B) = d_0{+}1$, 
the highest degree of the minimal generators in $B$
(see Theorem~\ref{theo:firstequivfregin}).
Moreover Claim~(1) follows from Theorem~\ref{theo:sectmatrixgenxk},
the hypothesis $\M_{S/B}(2,d_0{+}1)=\M_{S/B}(2,d_0{+}2)=0$,
and the fact that $B$ has no generator divisible by $x_3$
(again by Theorem~\ref{theo:firstequivfregin}).


On the other hand suppose (1) and (2) hold.
Then by Theorem~\ref{theo:sectmatrixgenxk}
$\M_{S/B}(3,d_0+1)=\M_{S/B}(3,d_0{+}2)$ 
implies that $\rgin(J(\A))$ has no minimal generators of degree
$d_0{+}2$ divisible by $x_3$ and hence, 
by Lemma~\ref{lem:gingensequiv}, 
it follows that it has no minimal generators of degree $d_0{+}2$.
By Lemma~\ref{lemm:genofalldeg}, 
it follows that $d_0{+}1 = \reg(B)$.
So $\M_{S/B}(3,d_0) = \M_{S/B}(3,d_0{+}1) = \M_{S/B}(3,d_0{+}2)$ 
implies that 
$\M_{S/B}(3,d{-}1) = \M_{S/B}(3,d)$ 
for all $d_0{+}1\le d\le\reg(B)$.

Hence Claim~(2) implies, by Theorem~\ref{theo:charcmhomideal},
that $B$ is Cohen-Macaulay of codimension $2$,
and we conclude that $\A$ is free.
\end{proof}

Similarly to Example~\ref{ex:recurring} we can consider the following.
\begin{Example} 
Consider the arrangement $\A$ in $\C^3$ defined by the equation $Q(\A)=xyz(x+y)(x-y)$.
Then the sectional matrix of $J(\A)$ is
$$\begin{array}{rcccccccccc} 
 _0 & _1 & _2 & _3  & _4 & \fbox{$_5$} & _6 & _7 & \dots\\
  1 & 1 & 1 & 1 & 0 & 0 & 0 & 0 & \dots\\
  1 & 2 & 3 & 4 & 2 & 1 & 0 & 0 & \dots\\
  1 & 3 & 6 & 10 & 12 & \fbox{$13$} & \fbox{$13$} & \fbox{$13$} & \dots
\end{array} 
$$
In this case, $d_0=5$,
$\M_{S/J(\A)}(3,5)=\M_{S/J(\A)}(3,6)=\M_{S/J(\A)}(3,7)=13$,
and $\M_{S/J(\A)}(3,5)= \sum_{d=0}^{d_0} \M_{S/J(\A)}(2,d)$.  
Hence $\A$ is free.

Similarly consider the arrangement $\A$ in $\C^3$ defined by the equation $Q(\A)=x(x+y-z)(x+z)(x+2z)(x+y+z)$.
Then the sectional matrix of its Jacobian ideal is
$$\begin{array}{rcccccccccc} 
 _0 & _1 & _2 & _3  & \fbox{$_4$} & _5 & _6 & _7 & \dots\\
  1 & 1 & 1 & 1 & 0 & 0 & 0 & 0 & \dots\\
  1 & 2 & 3 & 4 & 2 & 0& 0 & 0 & \dots\\
  1 & 3 & 6 & 10 & \fbox{$12$} & \fbox{$12$} & \fbox{$11$} & 11 & \dots
\end{array} 
$$
In this case, $d_0=4$ and $\M_{S/J(\A)}(3,4)=\M_{S/J(\A)}(3,5)=12$, but we have $\M_{S/J(\A)}(3,6)=11$.  Hence $\A$ is not free.
\end{Example}

\section{Hyperplane arrangements and resolutions}\label{sec:freearrres}
This section is devoted to prove some additional properties of $\rgin(J(\A))$ under the assumption that $\A$ is free.
In particular, our goal is to show that if $\A$ is free, then $\rgin(J(\A))$, and hence its sectional matrix, is combinatorially determined.
Moreover, we will describe how to compute the free resolution of $\rgin(J(\A))$ just from the exponents of $\A$, 
and, viceversa, how to compute the exponents of $\A$ from the degrees of the minimal generators of $\rgin(J(\A))$.

Before proceeding recall that, as seen in the construction of the proof of Theorem~\ref{theo:firstequivfregin}, we have the following
\begin{Remark}\label{rem:shapessicmcod2} Let $B$ be a strongly stable ideal of $K[x_1,\dots,x_l]$. If $B$ is Cohen-Macaulay of codimension $2$, then $$B=\ideal{x_1^{n-1}, x_1^{n-2}x_2^{\lambda_1},\dots, x_1x_2^{\lambda_{n-2}}, x_2^{\lambda_{n-1}}},$$ for some $0<\lambda_1<\lambda_2<\cdots<\lambda_{n-1}$.
\end{Remark}

By the definitions of reduction number and sectional matrix, we have the following

\begin{Remark}\label{rem:newlambdavalue} 
Let $\A=\{H_1, \dots, H_n\}$ be a central arrangement in $K^l$. Suppose that $\A$ is free and $\rgin(J(\A))=\ideal{x_1^{n-1},x_1^{n-2}x_2^{\lambda_1},\dots,x_2^{\lambda_{n-1}}}$.
 Then $\lambda_{n-1}=r_{l-2}(S/J(\A))+1$. 
Moreover, $\lambda_{n-1}$ is equal to the minimum $d\ge n-1$ such that $\M_{S/J(\A)}(n,d{+}1)=\sum_{i=n}^1\M_{S/J(\A)}(i,d)$.
\end{Remark}

In the next two results we make use of the exact sequence in Remark \ref{remark:resjaccmfree}, hence we suppose that $\A$ is also essential.
\begin{Proposition}\label{prop:computelastexp}
Let $\A=\{H_1, \dots, H_n\}$ be an essential and central arrangement in $K^l$. Suppose that $\A$ is free with exponents $(e_1,\dots, e_l)$ and 
$\rgin(J(\A))=\ideal{x_1^{n-1},x_1^{n-2}x_2^{\lambda_1},\dots,x_2^{\lambda_{n-1}}}$. 
Then $\lambda_{n-1}=e_l+n-2$.
\end{Proposition}

\begin{proof} 
By the exact sequence in Remark \ref{remark:resjaccmfree}, $\reg(J(\A))=e_l+n-2$. 
By Theorem~\ref{theo:regginequalideal}, $\reg(J(\A))$ coincides with the biggest degree of a minimal generator of $\rgin(J(\A))$. 
We conclude by Theorem~\ref{theo:firstequivfregin}.
\end{proof}

In general, given an ideal $I$ and its resolution, we cannot determine the resolution of $\rgin(I)$, see the last section of \cite{BPT2016}.
However, the following theorem shows that in the case of free arrangements we can. It shows that
 $\rgin(J(\A))$ is uniquely determined by the exponents of $\A$. 
In particular, it describes how to compute the Betti numbers of $\rgin(J(\A))$ from the Betti numbers of $J(\A)$.

Before stating the theorem, we recall the following result from \cite{eliahou1990minimal}, as described in Corollary 7.2.3 of \cite{herzog2011monomial}.

\begin{Proposition}\label{prop:elihkercrit} Let $B$ be a strongly stable ideal in $K[x_1,\dots,x_l]$. Then 
$$\beta_{i,i+j}(B)=\sum_{k=1}^l\binom{k-1}{i} m_{k,j},$$
where $m_{k,j}$ is the number of minimal generators of $B$ of degree $j$ such that the biggest variable that divides them is $x_k$.
\end{Proposition}

\begin{Theorem}\label{theo:resrginJ} 
Let $\A=\{H_1, \dots, H_n\}$ be an essential and central arrangement in $K^l$, with $l\ge2$. 
If $\A$ is free with exponents $(e_1,\dots,e_l)$ then $\rgin(J(\A))$ has free resolution
\begin{equation*}\label{eq:resfreeginjac}
0{\to}{\bigoplus_{j=n{-}1}^{n{+}e_l{-}2}}S(-j-1)^{\beta_{1,j+1}}{\to} {\bigoplus_{j=n{-}1}^{n{+}e_l{-}2}} S(-j)^{\beta_{0,j}}{\to} \rgin(J(\A)){\to}0,
\end{equation*}
where $\beta_{0,n-1}=\beta_{1,n}{+}1=l$ and $\beta_{1,j+1}=\beta_{0,j}=\#\{i~|~e_i{>} j{-}n{+}1\}$ for all $j\ge n$.
In particular, $\beta_{0,n-1}>\beta_{0,n}\ge\cdots\ge\beta_{0,n{+}e_l{-}2}$.
\end{Theorem}

\begin{proof} 
By Hilbert-Burch Theorem and Theorem~\ref{theo:firstequivfregin}, 
we have just to describe the connections between the exponents of $\A$ and the graded Betti numbers of $B=\rgin(J(\A))$.

In our situation, we have that $m_{1,n-1}=1$, $m_{1,j}=0$ for all $j\ne n-1$ and $m_{k,j}=0$ for all $k\ge3$. Hence, by Proposition \ref{prop:elihkercrit}, we get that $\beta_{0,j}(B)=m_{2,j}=\beta_{1,j{+}1}(B)$ for all $j\ge n$ and $\beta_{0,n{-}1}(B)=m_{1,n{-}1}+m_{2,n{-}1}=1+m_{2,n{-}1}=1+\beta_{1,n}(B)$. 

Furthermore, by the Cancellation Principle, we have that $\beta_{0,j}(J(\A))-\beta_{1,j}(J(\A))=\beta_{0,j}(B)-\beta_{1,j}(B)$. If $j\ge n$, $\beta_{0,j}(J(\A))=0$ and therefore, $\beta_{1,j}(B)=\beta_{0,j}(B)+\beta_{1,j}(J(\A))$. By the first part of the proof, $\beta_{0,j}(B)+\beta_{1,j}(J(\A))=\beta_{1,j{+}1}(B)+\beta_{1,j}(J(\A))$, and hence by iterating this process we can write  $\beta_{1,j}(B)=\sum_{k=0}^{n+e_l-1-j}\beta_{1,j+k}(J(\A))$. This shows that $\beta_{0,j}(B)=\beta_{1,j{+}1}(B)=\sum_{k=1}^{n+e_l-1-j}\beta_{1,j+k}(J(\A))$. Similarly, if $j=n-1$, then $\beta_{0,n-1}(B)=\beta_{1,n}(B){+}1=\sum_{k=1}^{e_l}\beta_{1,n-1+k}(J(\A)){+}1=l$. The statement follows from Remark \ref{remark:resjaccmfree}, since $\beta_{1,j}(J(\A))=\#\{i~|~n{+}e_i{-}1{=}j\}$.
\end{proof}

\begin{Remark} From Theorem~\ref{theo:resrginJ}, given $\A$ an essential, central and free arrangement in $K^l$, 
we have that $B = \rgin(J(\A)) \subset S = K[x_1,..,x_l]$ is $S$ or has exactly $n=\#\A$ generators, with exactly $l$ generators in degree $n-1$.
Moreover, there are $n-l$ generators in higher degrees, at least one in each degree up to the maximum, giving a bound of $2n -l-1$.
Hence, we have that  $\reg(J(\A))\le 2n-l-1$.
\end{Remark}

A direct consequence of the previous theorem and Theorem \ref{theo:firstequivfregin} is the following
\begin{Corollary}\label{corol:uniquerginfromfreeexp} Let $\A$ be an essential and central arrangement in $K^l$, with $l\ge2$. 
If $\A$ is free, then $\rgin(J(\A))$ is uniquely determined by the exponents of $\A$.
\end{Corollary}

\begin{Example} 
Consider the essential arrangement $\A$ in $\C^3$ with defining equation
$Q(\A)=xyz(x-y)$. A direct computation shows that $\A$ is free with exponents $(1,1,2)$. In fact, $J(\A)$ has a free resolution
$$0\to S(-4)\oplus S(-5) \to S(-3)^3\to J(\A)\to0$$
and the exponents can be computed using Remark \ref{remark:resjaccmfree}.
By Theorem~\ref{theo:resrginJ}, we have that $\beta_{0,3}=3$, $\beta_{1,4}=\beta_{0,3}-1=2$ and $\beta_{1,5}=\beta_{0,4}=\#\{i~|~e_i{>} 1\}=\#\{e_3\}=1$.
Thus, the resolution of $\rgin(J(\A))$ is
$$0 \to S(-4)^2\oplus S(-5) \to S(-3)^3\oplus S(-4)\to \rgin(J(\A))\to0.$$
Hence, from Theorem \ref{theo:firstequivfregin} it follows  that $\rgin(J(\A))=\ideal{x^3,x^2y,xy^2}+\ideal{y^4}$.
%
%
\end{Example}

Now we show that also the converse of Corollary \ref{corol:uniquerginfromfreeexp} holds true.
\begin{Proposition}\label{prop:fromgintoexp} Let $\A$ be an essential and central arrangement in $K^l$, with $l\ge2$.
If $\A$ is free, then the exponents of $\A$ are uniquely determined by $\rgin(J(\A))$.
\end{Proposition}
\begin{proof}  Assume $\A=\{H_1, \dots, H_n\}$. By assumption $\A$ is essential, hence, by Remark \ref{rem:assumptexpfree}, $e_1=1$. 
Moreover, since $\A$ is free, then by Theorem \ref{theo:firstequivfregin}, we can write $\rgin(J(\A))=\ideal{x_1^{n-1},x_1^{n-2}x_2^{\lambda_1},\dots,x_2^{\lambda_{n-1}}}$, 
for some $1\le\lambda_1<\lambda_2<\cdots<\lambda_{n-1}$. By Proposition \ref{prop:computelastexp}, $e_l=\lambda_{n-1}-n+2$.

With the notation of Theorem \ref{theo:resrginJ}, $\beta_{0,j}=\#\{i~|~\lambda_i+n-i-1=j\}$. Again by Theorem \ref{theo:resrginJ},
we have that $\#\{i~|~e_i=\alpha\}=\beta_{0,\alpha+n-2}-\beta_{0,\alpha+n-1}$ for all $\alpha\ge1$.

Notice that in this way we have uniquely identified the first $\sum_{j=n-1}^{\lambda_{n-1}}\beta_{0,j}-\beta_{0,j+1}=\beta_{0,n-1}-\beta_{0,\lambda_{n-1}}<l$
of the $e_i$'s. The remaining ones are now equal to $\lambda_{n-1}-n+2$.
\end{proof}

It is known that if $\A$ is free, then its exponents are combinatorially determined, see \cite{orlterao}. By Corollary \ref{corol:uniquerginfromfreeexp}, this allows us to have the following.
\begin{Corollary}\label{cor:freelatsamegin} 
Let $\A$ and $\A'$ be two free arrangements. Suppose that $\A$ and $\A'$ are lattice equivalent, then $\rgin(J(\A))=\rgin(J(\A'))$.
\end{Corollary}

The converse of the previous corollary is false.

\begin{Example}\emph{(cf. Example 2.61 \cite{orlterao})}
Consider the arrangements in $\C^3$, $\A=\{xyz(x-z)(x+z)(y-z)(y+z)=0\}$ and $\A'=\{xyz(x+y-z)(x+y+z)(x-y-z)(x-y+z)=0\}$. 
Then $\A$ and $\A'$ are both free arrangements with exponents $(1,3,3)$ and 
$\rgin(J(\A))=\rgin(J(\A'))=\ideal{x^6,x^5y,x^4y^2,x^3y^4,x^2y^5,xy^7,y^8}$. However, these two arrangements have non-equivalent lattices.
\end{Example}

The following example shows that Corollary~\ref{cor:freelatsamegin} is false if we do not assume that $\A$ and $\A'$ are free. 

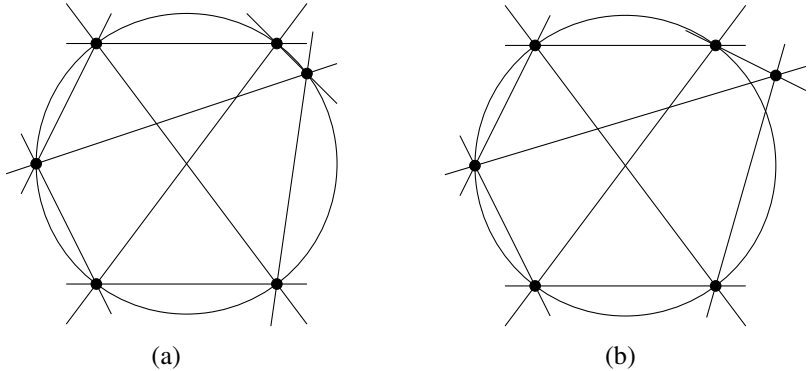
\begin{figure}[h!]
\centering
\subfigure[]
{
\begin{tikzpicture}[scale=0.4] 
\filldraw (-5,0) circle (5pt) node [] (1) {};
\filldraw (-3,-4) circle (5pt) node[](2){};
\filldraw (3,-4) circle (5pt) node[](3){};
\filldraw (4,3) circle (5pt) node[](4){};
\filldraw (3,4) circle (5pt) node[](5){};
\filldraw (-3,4) circle (5pt) node[](6){};
\draw (0,0) circle (5cm);
\draw[domain=-4:4] plot(\x,4);
\draw[domain=2:5] plot(\x,7-\x);
\draw[domain=2.8:4.2] plot(\x,7*\x-25);
\draw[domain=-4:4] plot(\x,-4);
\draw[domain=-5.5:-2.5] plot(\x,-10-2*\x);
\draw[domain=-5.5:-2.5] plot(\x,10+2*\x);
\draw[domain=-6:5] plot(\x,1/3*\x+5/3);
\draw[domain=-4:4] plot(\x,-4/3*\x);
\draw[domain=-4:4] plot(\x,4/3*\x);
\end{tikzpicture} 
}
\hspace{10mm}
\subfigure[]
{
\begin{tikzpicture}[scale=0.4]
\filldraw (-5,0) circle (5pt) node [] (1) {};
\filldraw (-3,-4) circle (5pt) node[](2){};
\filldraw (3,-4) circle (5pt) node[](3){};
\filldraw (5,3) circle (5pt) node[](4){};
\filldraw (3,4) circle (5pt) node[](5){};
\filldraw (-3,4) circle (5pt) node[](6){};
\draw (0,0) circle (5cm);
\draw[domain=-4:4] plot(\x,4);
\draw[domain=2:6] plot(\x,-1/2*\x+11/2);
\draw[domain=2.7:5.3] plot(\x,7/2*\x-29/2);
\draw[domain=-4:4] plot(\x,-4);
\draw[domain=-5.5:-2.5] plot(\x,-10-2*\x);
\draw[domain=-5.5:-2.5] plot(\x,10+2*\x);
\draw[domain=-6:6] plot(\x,3/10*\x+3/2);
\draw[domain=-4:4] plot(\x,-4/3*\x);
\draw[domain=-4:4] plot(\x,4/3*\x);
\end{tikzpicture} 
}
\caption{The arrangements of Example \ref{ex:nonfree}}\label{Fig:6pointsconic}
\end{figure}

\begin{Example}\label{ex:nonfree}
Consider the arrangements $\A=\{z(y-4z)(y+x-7z)(y-7x+25z)(y+4z)(y+2x{+}10z)(y-2x-10z)(3y-x-5z)(3y+4x)(3y-4x)=0\}$ and 
$\A'=\{z(y-4z)(2y+x-11z)(2y-7x+29z)(y+4z)(y+2x{+}10z)(y-2x-10z)(10y-3x-15z)(3y+4x)(3y-4x)=0\}$ in $\C^3$. 
We can see them as line arrangement in $\PPP^2$. See Figure \ref{Fig:6pointsconic}. Then, the first one consists of 10 lines that meet in exactly 6 triple points all 
sitting on the conic $\mathcal{C}=\{x^2+y^2-25z^2=0\}$, and the second one consists of 10 lines that meet in exactly 6 triple points 
but only 5 of them sit on the conic $\mathcal{C}$. Now, both $\A$ and $\A'$ are not free but $L(\A)\cong L(\A')$. 
A direct computation shows that $\rgin(J(\A))\ne\rgin(J(\A'))$.
\end{Example}

\begin{Remark} 
The statements of this section, and of sections \ref{sec:freearrgin} and \ref{sec:freearrsectmat} can be easily generalized to the case of reduced homogenous free divisors. For the statements on essential arrangements,  
we just need to require that the divisor $D$ is embedded in a space of minimal dimension, so that in $\Der(-\log D)$ there are no logarithmic
vector fields of degree $0$.
\end{Remark}

\section{From strongly stable ideals to free hyperplane arrangements}
Having in mind Theorem \ref{theo:resrginJ}, one could ask if given a Cohen-Macaulay strongly stable ideal $B$ of codimension 2, there  always exists a free hyperplane arrangement $\A$ such that $B=\rgin(J(\A))$. In general, the answer is no, see Example \ref{ex:nonfreearrattach} for more details. 
This section is devoted to characterize the class of strongly stable ideals for which we have a positive answer.

Clearly if $B=\ideal{1}$, then we can consider $\A=\{H\}$. Since we are looking for free hyperplane arrangements, in this section we consider only 
strongly stable ideals $B\subsetneq S=K[x_1,\dots,x_l]$ that are Cohen-Macaulay of codimension $2$.
Then $B$ has a free resolution of the type
\begin{equation*}\label{eq:resststcm2}
0\to\bigoplus_{j\ge2}S(-j)^{\beta_{1,j}}\to\bigoplus_{j\ge1}S(-j)^{\beta_{0,j}}\to B\to 0.
\end{equation*}
From now on, we will denote by $$d_{\min}=\min\{j~|~\beta_{0,j}\ne0\}\quad\text{ and }\quad d_{\max}=\max\{j~|~\beta_{0,j}\ne0\}.$$
\begin{Remark}
By Theorem \ref{theo:resrginJ}, if $\beta_{0,d_{\min}}\ne\dim(S)$, then there does not exist a free hyperplane arrangement $\A\subset K^l$ such that $B=\rgin(J(\A))$.
\end{Remark}
\begin{Example} 
Consider the strongly stable ideal $B=\ideal{x^3,x^2y^2,xy^4,y^6}$ in $S=K[x,y]$. 
Then $d_{\min}=3$ and $\beta_{0,d_{\min}}=1<2=\dim(S)$. Hence by the previous remark there does not exist 
a free hyperplane arrangement $\A\subset K^2$ such that $B=\rgin(J(\A))$.
\end{Example}
\begin{Remark}
By Theorem \ref{theo:firstequivfregin}, in $\rgin(J(\A))$ we have no ``holes". Hence 
if there exists $d_{\min}<j<d_{\max}$ such that $\beta_{0,j}=0$, then there does not exist a free hyperplane arrangement $\A\subset K^l$ such that $B=\rgin(J(\A))$.
\end{Remark}
\begin{Example}\label{ex:nonfreearrattach}
Consider the strongly stable ideal $B=\ideal{x_1^3, x_1^2x_2, x_1x_2^2, x_2^5}$ in $S=K[x_1,\dots,x_l]$, where $l\ge2$.
Then $d_{\min}=3$ and $d_{\max}=5$. However, since $B$ has no minimal generators of degree $4$, $\beta_{0,4}$ is $0$. 
Hence, by the previous remark, there does not exist 
a free hyperplane arrangement $\A\subset K^l$ such that $B=\rgin(J(\A))$, for any $l\ge2$.
\end{Example}
\begin{Remark} By Theorem \ref{theo:resrginJ}, if $\beta_{0,d_{\min}}\le\beta_{0,d_{\min}+1}$ or if $\beta_{0,j}<\beta_{0,j+1}$ for some $d_{\min}<j<d_{\max}$, then there does not exist a free hyperplane arrangement $\A\subset K^l$ such that $B=\rgin(J(\A))$.
\end{Remark}
\begin{Example} Consider the strongly stable ideal $B=\ideal{x_1^3, x_1^2x_2, x_1x_2^3, x_2^4}$ in $S=K[x_1,\dots,x_l]$, where $l\ge2$.
Then $d_{\min}=3$ and $d_{\max}=4$. Moreover, $2=\beta_{0,d_{\min}}=\beta_{0,d_{\min}+1}=\beta_{0,d_{\max}}$. Hence, 
then there does not exist a free hyperplane arrangement $\A\subset K^l$ such that $B=\rgin(J(\A))$, for any $l\ge2$.

Similarly, if we consider the ideal $B=\ideal{x_1^5, x_1^4x_2, x_1^3x_2^2, x_1^2x_2^4, x_1x_2^6,x_2^7}$ in $S=K[x_1,\dots,x_l]$, where $l\ge2$. Then we have
the same conclusion of before, since $1=\beta_{0,6}<\beta_{0,7}=2$.
\end{Example}
Before stating the main result of the section, we need the following construction.
\begin{Proposition}\label{prop:exptofreearr} Given $l-1$ integers such that $1\le e_2\le\cdots \le e_l$, then there exists an essential and central arrangement $\A$ in $K^l$ that is free with exponents $(1,e_2,\dots, e_l)$.
\end{Proposition}
\begin{proof} Consider the arrangement $\A$ in $K^l$ consisting of the following hyperplanes
$$\{x_1=0\},$$
$$\{\{x_1-\alpha_2x_2=0\}~|~\alpha_2\in\{1,\dots, e_2\}\},$$
$$\vdots$$
$$\{\{x_1-\alpha_lx_l=0\}~|~\alpha_l\in\{1,\dots, e_l\}\}.$$
By construction, the arrangement is essential, central and supersolvable (see \cite{orlterao} for the definition). By Theorem 4.58 in \cite{orlterao}, $\A$ is free with exponents $(1,e_2,\dots, e_l)$.
\end{proof}
\begin{Theorem}\label{theo:fromststcmtofreearr} Let $B$ be a Cohen-Macaulay strongly stable ideal in $K[x_1,\dots,x_l]$ of codimension $2$. Assume that the following conditions hold
\begin{enumerate} 
\item $\beta_{0,d_{\min}}=l$;
\item $\beta_{0,d_{\min}}>\beta_{0,d_{\min}+1}\ge\cdots\ge\beta_{0,d_{\max}}$.
\end{enumerate}
Then there exists a free hyperplane arrangement $\A\subset K^l$ such that $B=\rgin(J(\A))$. In particular, $\A$ has $d_{\min}+1$ hyperplanes.
\end{Theorem}
\begin{proof} 
Notice that, from the hypothesis, $\beta_{0,d_{\min}}-\beta_{0,d_{\min}+1}\ge1$. Define $e_i=1$, for all $i=1,\dots, \beta_{0,d_{\min}}-\beta_{0,d_{\min}+1}$. 
For all $j=d_{\min}+2,\dots, d_{\max}$, define $e_i=j-d_{\min}$ for all $i=\beta_{0,d_{\min}}-\beta_{0,j-1}+1,\dots, \beta_{0,d_{\min}}-\beta_{0,j}$.
Notice that by construction, the number of $e_i$ equal to $j-d_{\min}$ is $\beta_{0,j-1}-\beta_{0,j}$. 

In this way we have defined the first $\sum_{j=d_{\min}+1}^{d_{\max}}\beta_{0,j-1}-\beta_{0,j}$ of the $e_i$. By construction 
$$\sum_{j=d_{\min}+1}^{d_{\max}}\beta_{0,j-1}-\beta_{0,j}=\beta_{0,d_{\min}}-\beta_{0,d_{\max}}<l.$$
Define now the remaining $e_i$ equal to $d_{\max}-d_{\min}+1$. Notice that by construction, the number of $e_i$ equal to $d_{\max}-d_{\min}+1$ is $\beta_{0,d_{\max}}$. Notice now that by Remark \ref{rem:shapessicmcod2}, we have
\begin{eqnarray}\label{eq:sumexpisright}
\sum_{i=1}^le_i&=&\sum_{j=d_{\min}+1}^{d_{\max}}(j-d_{\min})(\beta_{0,j-1}-\beta_{0,j})+(d_{\max}-d_{\min}+1)\beta_{0,d_{\max}}= \nonumber \\
&=&\sum_{j=d_{\min}}^{d_{\max}}\beta_{0,j}=\#\{\text{~minimal generators of~} B\}=d_{\min}+1.
\end{eqnarray}
In this way we have constructed $l$ integers that satisfy the hypothesis of Proposition \ref{prop:exptofreearr}, 
and hence there exists an essential arrangement $\A$ in $K^l$ that is free with exponents $(e_1=1,e_2,\dots, e_l)$. 
Now, by construction, Theorem \ref{theo:resrginJ} and equality \eqref{eq:sumexpisright}, $B$ and  $\rgin(J(\A))$ have the same resolution. 
By Corollary \ref{corol:uniquerginfromfreeexp}, we have that $B=\rgin(J(\A))$.
\end{proof}
\begin{Example}
Consider the ideal $B=\ideal{x^6, x^5y, x^4y^2, x^3y^4, x^2y^5,xy^7,y^9}$ in $S=K[x,y,z]$.
Then $d_{\min}=6$ and $d_{\max}=9$, and $\beta_{0,6}=3$, $\beta_{0,7}=2$ and $\beta_{0,8}=\beta_{0,9}=1$. Using the construction of Theorem \ref{theo:fromststcmtofreearr},
we obtain $(e_1,e_2,e_3)=(1,2,4)$. Consider now the arrangement $\A$ in $K^3$ defined by $Q=x(x-y)(x-2y)(x-z)(x-2z)(x-3z)(x-4z)$, then $B=\rgin(J(\A))$.
\end{Example}

Putting together Theorems \ref{theo:fromststcmtofreearr} and \ref{theo:resrginJ}, we obtain the following characterization for the $\rgin$ associated to essential, central and free
hyperplane arrangements.
\begin{Corollary} Let $B$ a strongly stable ideal in $K[x_1,\dots,x_l]$. There exists an essential, central and free arrangement $\A$ of $n$ hyperplanes such that $B=\rgin(J(\A))$ if and only if $B$ is minimally generated by $$x_1^{n-1}, x_1^{n-2}x_2^{\lambda_1},\dots, x_1x_2^{\lambda_{n-2}}, x_2^{\lambda_{n-1}}$$ with $\#\{i~|~\lambda_i=i\}\ge\#\{i~|~\lambda_i=i+1\}\ge\cdots\ge\#\{i~|~\lambda_i=i+\lambda_{n-1}-n+1\}$.
\end{Corollary}

\paragraph{\textbf{Acknowledgements}} 
During the preparation of this paper the third author was supported by the MEXT grant for Tenure Tracking system. The authors thanks D. Suyama and S. Tsujie for suggesting the construction in Proposition \ref{prop:exptofreearr}.

All the computations in the paper are done using CoCoA, see \cite{CoCoALib}, \cite{AbbottBigatti2016} and \cite{COCOA}.

\bibliography{bibliothesis}{}

\begin{thebibliography}{10}

\bibitem{CoCoALib}
J.~Abbott and A.M. Bigatti.
\newblock {CoCoALib}: a {C}++ library for doing {C}omputations in {C}ommutative
  {A}lgebra.
\newblock Available at \texttt{http://cocoa.dima.unige.it/cocoalib}, 2016.

\bibitem{AbbottBigatti2016}
J.~Abbott and A.M. Bigatti.
\newblock Gr\"obner bases for everyone with {C}o{C}o{A}-5 and {C}o{C}o{AL}ib.
\newblock arXiv:1611.07306, 2016.

\bibitem{COCOA}
J.~Abbott, A.M. Bigatti, and L.~Robbiano.
\newblock {CoCoA}: a system for doing {C}omputations in {C}ommutative
  {A}lgebra.
\newblock Available at \texttt{http://cocoa.dima.unige.it}.

\bibitem{abe2013free}
T.~Abe and M.~Yoshinaga.
\newblock Free arrangements and coefficients of characteristic polynomials.
\newblock {\em Mathematische Zeitschrift}, 275(3-4):911--919, 2013.

\bibitem{bayer1987criterion}
D.~Bayer and M.~Stillman.
\newblock A criterion for detecting m-regularity.
\newblock {\em Inventiones mathematicae}, 87(1):1--11, 1987.

\bibitem{BPT2016}
A.~Bigatti, E.~Palezzato, and M.~Torielli.
\newblock Extremal behavior in sectional matrices.
\newblock {\em arXiv: 1702.03292}, 2017.

\bibitem{bigatti1997borel}
A.~Bigatti and L.~Robbiano.
\newblock Borel sets and sectional matrices.
\newblock {\em Annals of Combinatorics}, 1(1):197--213, 1997.

\bibitem{eliahou1990minimal}
S.~Eliahou and M.~Kervaire.
\newblock Minimal resolutions of some monomial ideals.
\newblock {\em Journal of Algebra}, 129(1):1--25, 1990.

\bibitem{galligo1974propos}
A.~Galligo.
\newblock A propos du th{\'e}oreme de pr{\'e}paration de {W}eierstrass.
\newblock In {\em Fonctions de plusieurs variables complexes}, pages 543--579.
  Springer, 1974.

\bibitem{herzog2011monomial}
J.~Herzog and T.~Hibi.
\newblock {\em Monomial ideals}.
\newblock Springer, 2011.

\bibitem{orlterao}
P.~Orlik and H.~Terao.
\newblock {\em Arrangements of hyperplanes}, volume 300 of {\em Grundlehren der
  Mathematischen Wissenschaften [Fundamental Principles of Mathematical
  Sciences]}.
\newblock Springer-Verlag, Berlin, 1992.

\bibitem{saito}
K.~Saito.
\newblock Theory of logarithmic differential forms and logarithmic vector
  fields.
\newblock {\em J. Fac. Sci. Univ. Tokyo Sect. IA Math.}, 27(2):265--291, 1980.

\bibitem{schulze2012freeness}
M.~Schulze.
\newblock Freeness and multirestriction of hyperplane arrangements.
\newblock {\em Compositio Mathematica}, 148(03):799--806, 2012.

\bibitem{terao1980arrangementsI}
H.~Terao.
\newblock Arrangements of hyperplanes and their freeness {I}.
\newblock {\em J. Fac. Sci. Univ. Tokyo Sect. IA Math.}, 27(2):293--312, 1980.

\bibitem{yoshinaga2004characterization}
M.~Yoshinaga.
\newblock Characterization of a free arrangement and conjecture of {E}delman
  and {R}einer.
\newblock {\em Inventiones mathematicae}, 157(2):449--454, 2004.

\bibitem{yoshinaga2005freeness}
M.~Yoshinaga.
\newblock On the freeness of 3-arrangements.
\newblock {\em Bulletin of the London Mathematical Society}, 37(01):126--134,
  2005.

\bibitem{ziegler1986multiarrangements}
G.~M. Ziegler.
\newblock Multiarrangements of hyperplanes and their freeness.
\newblock {\em Contemporary Math}, 90:345--358, 1986.

\end{thebibliography}
\bibliographystyle{plain}

\end{document}